\documentclass[letterpaper, 10pt, conference]{ieeeconf}  

\IEEEoverridecommandlockouts                              
\usepackage{epsfig}
\usepackage{amsmath}
\usepackage{amssymb}
\usepackage{indentfirst}
\usepackage{cite}
\usepackage{bm}
\usepackage{graphicx}           
\usepackage{psfrag}             
\usepackage{subfigure}          
\usepackage{empheq}
\usepackage{enumerate}

\usepackage{amsthm}

\usepackage{tikz}
\usetikzlibrary{arrows}

\newenvironment{proof}[1][Proof]%
  {\smallskip\par\noindent\textbf{#1\,:\ }}%
  {\hspace*{\fill} \rule{6pt}{6pt}\smallskip}
\newenvironment{proof*}[1][Proof]%
  {\medskip\par\noindent\textbf{#1\,:\ }}%

\newtheorem{remark}{\textbf{Remark}}

\newtheorem{assumption}{Assumption}
\newtheorem{theorem}{Theorem}

\usepackage{color}
\definecolor{gray}{RGB}{128,128,128}

\newcommand{\red}[1]{{\color{black} #1}}

\title{\red{Distributed Robust Dynamic Average Consensus with \\ Dynamic Event-Triggered Communication}}

\author{Jemin George, Xinlei Yi and Tao Yang
\thanks{Jemin~George is with the U.S. Army Research Laboratory, Adelphi, MD 20783, USA.
{\tt\small jemin.george.civ@mail.mil}}%
\thanks{Xinlei~Yi is with the ACCESS Linnaeus Center, School of Electrical Engineering and Computer Science, Royal Institute of Technology, 100 44 Stockholm, Sweden.
{\tt\small xinleiy@kth.se}}%
\thanks{Tao~Yang with the Department of Electrical Engineering, University of North Texas, Denton, TX 76203 USA.
\red{{\tt\small Tao.Yang@unt.edu}}}%
}

\allowdisplaybreaks[4]

\begin{document}


\maketitle
\thispagestyle{empty}
\pagestyle{empty}

\begin{abstract}
This paper presents the formulation and analysis of a fully distributed dynamic event-triggered communication based robust dynamic average consensus algorithm. Dynamic average consensus problem involves a networked set of agents estimating the time-varying average of dynamic reference signals locally available to individual agents. We propose an asymptotically stable solution to the dynamic average consensus problem that is robust to network disruptions. Since this robust algorithm requires continuous communication among agents, we introduce a novel dynamic event-triggered communication scheme to reduce the overall inter-agent communications. It is shown that the event-triggered algorithm is asymptotically stable and free of Zeno behavior. Numerical simulations are provided to illustrate the effectiveness of the proposed algorithm.
\end{abstract}

\section{Introduction}\label{sec:intro}

Consider a set of \red{$n$} networked agents, each with its own reference signal $\phi_i(t) \red{\in} \mathbb{R}^r$. The dynamic average consensus problem involves designing distributed algorithms that would allow the agents to locally estimate the time-varying average $\bar{\phi}(t) \triangleq  \frac{1}{n} \sum\limits_{i=1}^{n} \phi_i(t)$. This estimator design problem has numerous applications in multi-agent systems~\cite{George,George2018ICASSP}. For example, in the distributed optimization problem $ \min\limits_\red{x \in \mathbb{R}^r}\sum\limits_{i=1}^{n} f_i(x)$, $\phi_i(t) = \nabla f_i(\cdot)$, in distributed estimation problem, $\phi_i(t)$s are local weighted measurement-residuals, and in multi-agent coordination problems such as containment control, $\phi_i(t)$s are the leader trajectories. Thus dynamic average consensus is at the heart of numerous network applications\red{,} such as distributed learning, distributed sensor fusion, formation control, distributed optimization, and distributed mapping.

The main difficulty in designing distributed solutions to dynamic average consensus problem is the lack of access to any error signals. To be \red{more} specific, if $x_i(t)$ is the $i^{\text{th}}$-node's estimate of $\bar{\phi}(t)$, then none of the nodes have access to the average-consensus error $\tilde{x}_i\red{(t)}= x_i(t)-\bar{\phi}(t)$, thus rendering the traditional feedback-control techniques obsolete. Therefore solutions to dynamic average consensus problem was first proposed for reference signals with steady-state values~\cite{Spanos2005} or slowly varying reference signals \cite{Freeman:CDC2006}. Assuming access to the dynamics that generate the reference signal, an internal model based dynamic average consensus algorithms are presented in \cite{Bai10,Bai2016}. Assuming access to the time derivatives of the reference signals, a dynamic average consensus algorithm built on singular perturbation theory is given in \cite{Kia2013}. However in real world applications, it is not reasonable to assume knowledge of the reference signal dynamics or presume access to its time derivatives.

Even though the agents don't have direct access to the error signal $\tilde{x}_i(t)$, they can calculate the local difference or disagreement in the error, i.e., $\tilde{x}_i(t) - \tilde{x}_j(t)$. Thus the error signal is such that it sum to zero, then the dynamic average consensus problem is solved if the agents reach consensus on $\tilde{x}_i$. In other words, if $\sum\limits_{i=1}^{n} \tilde{x}_i(t) = 0$ and $\tilde{x}_i(t) = \tilde{x}_j(t)$ for all $(i,j)$ pairs, then $\tilde{x}_i(t)=0$ for all $i$. Therefore there exists several solutions to dynamic average consensus problems where an estimator is designed such that the estimator structure along with an initialization requirement provides the zero-sum condition $\tilde{x}_i(t)=0$ while the inputs to the estimator are selected such that the agents reach consensus on the error signal. Examples of such algorithms include the nonlinear dynamic average consensus estimators for reference signals with bounded derivative given in \cite{Nosrati2012}, \cite{Kia2014}, and \cite{Kia2015a}. The algorithms in \cite{Nosrati2012}, \cite{Kia2014}, and \cite{Kia2015a} are shown to yield bounded average-consensus error even for a directed network, but the error bounds are proportional to the upper bound on the time derivatives of the reference signals.  Besides the continuous-time algorithms discussed so far, there also exist several discrete-time dynamic average consensus algorithms \cite{Zhu2010, Yuan2012, Montijano2014b, Scoy2015, VanScoy2015CDC, Franceschelli2016}.

While the dynamic average consensus problem focus on designing estimators, a combined estimator and controller design problem to estimate and track the time-varying average signal is studied under the name \emph{distributed average tracking} (distributed average tracking)~\cite{WeiRen2011ACC,WeiRen2012TAC, WeiRen2013CCC, WeiRen2014ACC, WeiRen2015ACC, WeiRen2015TAC2, WeiRen2015TAC3, WeiRen2016ACC, Zhao2017, WeiRen2017ACC}. The problem formulation in distributed average tracking consists of assuming a particular dynamic-model for individual agents and then designing a distributed control law that allows the agents to track the time-varying average. The main drawback to considering a combined estimator/controller solution is that it is tailored to specific node dynamics and therefore only valid for the assumed dynamic-model. As a result, there exist numerous distributed average tracking solutions to the same average-consensus problem involving agents with single-integrator dynamics~\cite{WeiRen2011ACC,WeiRen2012TAC}, double-integrator dynamics~\cite{WeiRen2015ACC,WeiRen2015TAC3}, \textsc{E}uler-\textsc{L}agrange dynamics~\cite{WeiRen2014ACC, WeiRen2015TAC2}, known linear dynamics~\cite{WeiRen2013CCC}, nonlinear dynamics~\cite{WeiRen2016ACC}, heterogeneous dynamics~\cite{WeiRen2017ACC}, and so on. Furthermore, the combined estimator/controller solution limits the utility of such algorithms for numerous network applications such as distributed optimization. Besides, if the agents are able to estimate the time-varying average, say using a DAC estimator, then the control design problem is often trivial.

Dynamic average consensus algorithms in \cite{Nosrati2012}, \cite{Kia2014}, \cite{Kia2015a}, and \cite{Kia2015b} all require a specific initialization of its variables to satisfy the condition $\sum\limits_{i=1}^{n} \tilde{x}_i(t_0) = 0$. This requirement seems benign at first because it can be easily satisfied by selecting $x_i(t_0) = \phi_i(t_0)$ for all $i$. However, when an agent leaves the network or when the network split into several small subgraphs, the condition $\sum\limits_{i=1}^{n} \tilde{x}_i(t_0) = 0$ is violated. This result in a nonzero steady-state error unless all the nodes reinitialize the algorithm after every such network disruption. This algorithm sensitivity to initialization, typically referred to as the lack of robustness to \emph{initialization errors}, is an issue in most of the distributed average tracking approaches \cite{WeiRen2011ACC,WeiRen2012TAC,WeiRen2015ACC,WeiRen2015TAC3}. Currently, no systematic solution to this problem that does not sacrifice algorithm performance or introduce stringent assumptions on the reference signal or its dynamics exists.

The continuous-time solutions given in \cite{George2017ACC,Nosrati2012,Kia2015a,WeiRen2011ACC,WeiRen2012TAC, WeiRen2013CCC, WeiRen2014ACC, WeiRen2015ACC, WeiRen2015TAC2, WeiRen2015TAC3, WeiRen2016ACC, Zhao2017, WeiRen2017ACC} all assume continuous communication among agents. This is not a reasonable assumption especially if the agents are interacting via wireless communication network. Even though discrete time algorithms are more docile to implementation, none of the discrete time algorithms can guarantee zero steady-state error for the types of reference signals considered here. Furthermore, the use of a fixed communication step-size in discrete time algorithms can be a wasteful use of the network resources. Distributed event-triggered communication provides a way to address some of these challenges by locally designing inter-agent communication times in an opportunistic manner. Thus, instead of communicating continuously or periodically, the designed communication times or event times, allows the agents to determine when to communicate based on a specific triggering mechanism. Thus far references \cite{Kia2014} and \cite{Kia2015b} constitutes the only two literature on distributed event-triggered dynamic average mechanism.


Here we first present a dynamic average consensus algorithm that is robust to initialization errors (section \ref{sec:RobustDAC}). The robust algorithm given in section \ref{sec:RobustDAC} makes use of an adaptive gain which removes the explicit use of any upper bounds on reference signals or its time-derivative in the algorithm. We then present a distributed event triggered version of the algorithm in section \ref{sec:ETRobustDAC} which make use of a dynamic triggering mechanism. The event-triggered algorithm is shown to provide asymptotic convergence as well as free of Zeno\footnote{For continuous-time multi-agent systems, Zeno behavior means that there are infinite number of event triggers in a finite time interval~\cite{Johansson99}.} behavior. Compared to existing results, the proposed algorithm is novel in the following sense:
\begin{itemize}
  \item The proposed event-triggered algorithm is robust to network disruptions since it does not require any specific initialization criteria (see above discussion for details).
  \item The proposed algorithm can theoretically guarantee zero steady state error.
  \item The proposed triggering laws involve internal dynamic variables which play an essential role in guaranteeing that the triggering time sequence does not exhibit Zeno behavior.
\end{itemize}

\section{Preliminaries}\label{sec:pre}

\subsection*{Notation}

Let $\mathbb{R}^{n\times m}$ denote the set of $n\times m$ real matrices. An $n\times n$ identity matrix is denoted as $I_n$ and $\mathbf{1}_n$ denotes an $n$-dimensional vector of all ones. Let $\mathbb{R}^n_{\mathbf{1}}$ denote the set of all $n$-dimensional vectors of the form $\kappa\mathbf{1}_n$, where $\kappa\in\mathbb{R}$. For two vectors $\mathbf{x} \in \mathbb{R}^n$ and $\mathbf{y} \in \mathbb{R}^n$, $\mathbf{x} \geq  \mathbf{y} \,\,\left(\mathbf{x} \leq \mathbf{y} \right)$ implies $x_i \geq  y_i, \,\, \left({x}_i \leq {y}_i \right)$, $\forall\,i\in\{1,\ldots,n\}$. The absolute value of a vector is given as $|\mathbf{x}| = \begin{bmatrix} |x_1| & \ldots & |x_n|\end{bmatrix}^T$. Let $\text{sgn}\{\cdot\}$ denote the signum function, defined as
$$\text{sgn}\{x\} \triangleq \left\{
\begin{array}{ll}
+1, & \hbox{if $x>0$;} \\
0, & \hbox{if $x=0$;}  \\
 -1, & \hbox{if $x<0$,}
\end{array}
\right.
$$
and $\forall\,\mathbf{x}\in\mathbb{R}^n$, $\text{sgn}\{\mathbf{x}\} \triangleq \begin{bmatrix} \text{sgn}\{x_1\} & \ldots & \text{sgn}\{x_n\}\end{bmatrix}^T$. For $p\in[1,\,\infty]$, the $p$-norm of a vector $\mathbf{x}$ is denoted as $\left\| \mathbf{x} \right\|_p$. For matrices $A \in \mathbb{R}^{m\times n}$ and $B \in \mathbb{R}^{p \times q}$, $A \otimes B \in \mathbb{R}^{mp \times nq}$ denotes their Kronecker product.

\subsection*{Network Model}

For a \emph{connected undirected} graph $\mathcal{G}\left(\mathcal{V},\mathcal{E}\right)$ of order $n$, $\mathcal{V} \triangleq \left\{v_1, \ldots, v_n\right\}$ represents the agents or nodes. The communication links between the agents are represented as $\mathcal{E} \triangleq \left\{e_1, \ldots, e_{\ell}\right\} \subseteq \mathcal{V} \times \mathcal{V}$. Here each undirected edge is considered as two distinct directed edges and the edges are labeled such that they are grouped into incoming links to nodes $v_1$ to $v_n$. Let $\mathcal{I}$ denote the index set $\{1,\ldots,n\}$ and $\forall i \in \mathcal{I}$; let $\mathcal{N}_i \triangleq \left\{v_j \in \mathcal{V}~:~(v_i,v_j)\in\mathcal{E}\right\}$ denote the set of neighbors of node $v_i$. Let $A \triangleq \left[a_{ij}\right]\in \{0,1\}^{n\times n}$ be the \emph{adjacency matrix} with entries $a_{ij} = 1 $ if $(v_i,v_j)\in\mathcal{E}$ and zero otherwise. Define $\Delta \triangleq \text{diag}\left(A\mathbf{1}_n\right)$ as the degree matrix associated with the graph and $L \triangleq \Delta - A$ as the graph \emph{Laplacian}. The \emph{incidence matrix} of the graph is defined as $B = \left[b_{ij}\right]\in \left\{-1,0,1\right\}^{n\times \ell}$, where $b_{ij} = -1$ if edge $e_j$ leaves node $v_i$, $b_{ij} = 1$ if edge $e_j$ enters node $v_i$, and $b_{ij} = 0$ otherwise.

For the connected undirected graph $\mathcal{G}\left(\mathcal{V},\mathcal{E}\right)$, $L$ is a positive semi-definite matrix with one eigenvalue at 0 corresponding to the eigenvector $\mathbf{1}_n$. Since each undirected edge is considered as two distinct directed edges, we have $L = \frac{1}{2}BB^T$. Furthermore, we have $M \triangleq \left(I_n - \displaystyle\frac{1}{n}\mathbf{1}_n\mathbf{1}^T_n\right) = L \left( L \right)^+ = {BB}^T\left({BB}^T\right)^+ =$  $B\left(B^TB\right)^+B^T$, where $\left(\cdot\right)^+$ denotes the generalized inverse~(Lemma~3~\cite{Gutman04}).

\begin{remark}
For all $\mathbf{x}\in\mathbb{R}^n$, such that $\mathbf{1}^T_n\mathbf{x} = 0$, we have $\mathbf{x}^TL\left( L \right)^+\mathbf{x} =  \mathbf{x}^T\mathbf{x} $.
\end{remark}

\section{Problem Formulation}\label{sec:Problem}

Let ${\phi}_i(t) \in \mathbb{R}^r$ denote the $i^{\text{}th}$-node's ($v_i$'s) reference signal at time $t$. The dynamic average consensus problem involves each agent estimating the time-varying signal
\begin{align}
\bar{{\phi}}(t) = \frac{1}{n} \sum_{i=1}^{n}\,{\phi}_i(t) = \red{\frac{1}{n} \left(\mathbf{1}^T_n \otimes I_r\right) \bm{\phi}(t),}
\end{align}
where \red{$n$ is the number of agents, $r$ is the size of reference signals,} and $\bm{\phi}(t) \in$ $\mathbb{R}^{n r} \triangleq \begin{bmatrix} {\phi}^T_1(t) & \ldots & {\phi}^T_n(t)\end{bmatrix}^T$. Let $\dot{\bm{\phi}}(t) \triangleq \begin{bmatrix} \dot{{\phi}}^T_1(t) & \ldots & \dot{{\phi}}^T_n(t)\end{bmatrix}^T$. Now we make following standing assumptions:

\begin{assumption}\label{Assump:Graph}
The interaction topology of $n$ networked agents is given as a connected undirected graph $\mathcal{G}\left(\mathcal{V},\mathcal{E}\right)$.
\end{assumption}

\begin{assumption}\label{Assump:Phi}
For any two connected agents, the local difference in signals ${\phi}_i(t)$ and their derivatives $\dot{{\phi}}_i(t)$ are bounded such that there exist bounds $\varphi$ and $\dot{\varphi}$ that satisfy
\begin{align}
  &\sup_{\substack{t\in[0,\infty)\\\forall\,i,j:(v_i,v_j)\in\mathcal{E}}} \| \phi_i(t) - \phi_j(t)\|_{\infty} \leq \varphi < \infty, \qquad \text{and} \label{Eq:Phi:Assump}\\
  &\sup_{\substack{t\in[0,\infty)\\\forall\,i,j:(v_i,v_j)\in\mathcal{E}}} \| \dot{\phi}_i(t) - \dot{\phi}_j(t)\|_{\infty} \leq \dot{\varphi} < \infty. \label{Eq:DotPhi:Assump}
\end{align}
\end{assumption}
\vspace{0.25cm}
Note that Assumption \ref{Assump:Phi} is less \red{strict} than assuming absolute bounds on signals ${\phi}_i(t)$ and their derivatives $\dot{{\phi}}_i(t)$. Using vector notation \eqref{Eq:Phi:Assump} and \eqref{Eq:DotPhi:Assump} can be written as
\begin{align}
  &\sup_{\substack{t\in[0,\infty)}} \left\| \left( B^T \otimes I_r \right) \bm{\phi}(t) \right\|_{\infty} \leq \varphi , \qquad \text{and} \label{Eq:Phi:Assump1}\\
  &\sup_{\substack{t\in[0,\infty)}} \left\| \left( B^T \otimes I_r \right) \dot{\bm{\phi}}(t) \right\|_{\infty} \leq \dot{\varphi} . \label{Eq:DotPhi:Assump1}
\end{align}

\section{Robust Dynamic Average Consensus Algorithm}\label{sec:RobustDAC}

\subsection{Robust Algorithm}

Let $x_i(t) \in \mathbb{R}^{r}$ denote node $v_i$'s estimate of $\bar{\phi}(t)$. Here we propose the following robust dynamic average consensus algorithm:
\begin{subequations}\label{RDATsys1}
\begin{align}
  \dot{\mathbf{z}}(t) &=  -\gamma\,\mathbf{z}(t) + \mathbf{u}(t),\quad \mathbf{z}(t_0) = \mathbf{z}_0,\label{RDATsys1a}\\
  \mathbf{x}(t) &= \mathbf{z}(t) + \bm{\phi}(t)\label{RDATsys1b},
\end{align}
\end{subequations}
where $\gamma>0$ is a positive constant, $\mathbf{x}(t)$ $\in$ $\mathbb{R}^{nr}$ $\triangleq$ $\begin{bmatrix} x^T_1(t) & \ldots & x^T_n(t)\end{bmatrix}^T$ \red{is} the estimate of $\bar{\phi}(t)$ for the entire network, $\mathbf{z}(t)$ $\in$ $\mathbb{R}^{n r}$ $\triangleq$ $\begin{bmatrix} z^T_1(t) & \ldots & z^T_n(t)\end{bmatrix}^T$ is the internal state of the estimator for the entire network, and $\mathbf{u}(t)$ is the input that needs to be designed.

Let $\tilde{\mathbf{x}}(t)\triangleq\mathbf{x}(t)-\mathbf{1}_n \otimes \bar{\phi}(t)$ denote the dynamic average consensus error. Now the error dynamics can be written as
\begin{align*}
  \dot{\tilde{\mathbf{x}}}(t) &= \dot{\mathbf{z}}(t) + \dot{\bm{\phi}}(t) - \frac{1}{n} \left( \mathbf{1}_n\mathbf{1}^T_n \otimes I_r\right) \dot{\bm{\phi}}(t),\\
  &= -\gamma\mathbf{z}(t) + \mathbf{u}(t) + \left( M \otimes I_r\right) \dot{\bm{\phi}}(t).
\end{align*}
Now adding and subtracting $\gamma \left(M\otimes I_r\right) {\bm{\phi}} (t)$ yields
\begin{align}\label{DACError}
  \dot{\tilde{\mathbf{x}}}(t) & =  -\gamma \tilde{\mathbf{x}}(t) + \mathbf{u}(t) + \left(M\otimes I_r\right) \left( \dot{\bm{\phi}} (t) +\gamma  {\bm{\phi}} (t)\right).
\end{align}

\red{\subsection{Convergence Result}}
The following theorem illustrates how to select the inputs $\mathbf{u}(t)$ such that the average consensus-error asymptotically converges to zero.

\begin{theorem}\label{Theorem01}
Given Assumptions~\ref{Assump:Graph} and~\ref{Assump:Phi}, the robust dynamic average consensus algorithm in \eqref{RDATsys1} guarantees that the average consensus error, $\tilde{\mathbf{x}}(t)$, asymptotically decays to zero for any initial condition $\mathbf{z}_0$, if the estimator input $\mathbf{u}(t)$ is selected as
\begin{equation}\label{Eqn:u1}
\mathbf{u}(t) = - \left(B \otimes I_r\right)\,K(t) \, \text{sgn}\left\{\left(B^T \otimes I_r\right){\mathbf{x}}(t)\right\},
\end{equation}
where $K(t) \in \mathbb{R}^{r\ell\times r\ell}$ is a diagonal gain matrix with diagonal entries $\kappa_j(t)$, $j$ $\in$ $\{1,\ldots,r\ell\}$ updated according to
\begin{equation}\label{Eqn:kappa1}
\dot{\kappa}_j(t) = |y_j(t)|, \quad \kappa_j(t_0) \geq 1, 
\end{equation}
and $\mathbf{y}(t) \in \mathbb{R}^{r\ell} \triangleq  \begin{bmatrix} y_1(t) & \ldots & y_{r\ell}(t)\end{bmatrix}^T$ is defined as $\mathbf{y}(t) = \left( B^T \otimes I_r \right)\mathbf{x}(t)$.
\end{theorem}

\begin{proof}
\red{The proof} consists of two \red{steps}.
The first \red{step is to show} that the algorithm~\eqref{RDATsys1} exponentially satisfies the zero-sum condition $ \sum\limits_{i=1}^{n} \tilde{x}_i(t) = \mathbf{0}_r$.
The second \red{step is to show} that $\tilde{x}_i(t)$ asymptotically reach consensus.
If agents asymptotically reach consensus on $\tilde{x}_i(t)$s and $\tilde{x}_i(t)$s satisfy the zero-sum condition, then $\lim\limits_{t\rightarrow\infty} \tilde{x}_i(t) = \mathbf{0}_r$ for all $i\in\mathcal{I}$.

Left multiplying the error $\tilde{\mathbf{x}}(t)$ $=$ $\mathbf{z}(t)$ $+$ $\left(M\otimes I_r\right)\bm{\phi}(t)$ with $\left(\mathbf{1}^T_n \otimes I_r\right)$ yields $ \sum\limits_{i=1}^{n} \tilde{x}_i(t) = \sum\limits_{i=1}^{n} z_i(t)$, for all $t$. Now taking the time-derivative by substituting \eqref{Eqn:u1} in \eqref{RDATsys1a} and using the fact that $\mathbf{1}^TB = 0$ yields
\begin{align*}
  \frac{d}{dt} \left( \sum\limits_{i=1}^{n} z_i(t) \right) = -\gamma \sum\limits_{i=1}^{n} z_i(t).
\end{align*}
Thus $\sum\limits_{i=1}^{n} z_i(t)$, and therefore $\sum\limits_{i=1}^{n} \tilde{x}_i(t)$ is exponentially decreasing to $\mathbf{0}_r$ with the rate $\gamma$. This concludes the first \red{step} of the proof.

The second \red{step is to show} $\left( B^T \otimes I_r \right)\tilde{\mathbf{x}}(t)$ asymptotically decays to zero, i.e., the agents reach consensus on ${\tilde{\mathbf{x}}}(t)$. Note \red{that} $\left( B^T \otimes I_r \right)\tilde{\mathbf{x}}(t) = \left( B^T \otimes I_r \right)\mathbf{x}(t) = \mathbf{y}(t)$. Thus \red{it is equivalent to show} the asymptotic stability of
\begin{align*}
\dot{\mathbf{y}}(t) &= \left(B^T \otimes I_r \right) \left[-\gamma \mathbf{z}(t) +  \mathbf{u}(t) + \dot{\bm{\phi}}(t) \right],\\
&= \left(B^T \otimes I_r \right) \left[-\gamma \tilde{\mathbf{x}}(t) +  \mathbf{u}(t) + \dot{\bm{\phi}}(t) + \gamma \bm{\phi}(t) \right],
\end{align*}
where we used the equality condition $B^TM = B^T$.

Consider a nonnegative function of the form
$$V = \frac{1}{2} {\mathbf{y}}^T(t) \left( \left(B^TB\right)^+ \otimes I_r\right) {\mathbf{y}}(t) + \frac{1}{2} \sum^{r\ell}_{j=1} \left(\kappa_j(t)-\kappa^*\right)^2,$$
where $\kappa^*$ is a constant to be specified. Taking the time derivative of $V$ yields
\begin{align*}
\dot{V} &= {\mathbf{x}}^T(t) \left( B \left(B^TB\right)^+ B^T \otimes I_r\right) \bigg[-\gamma \tilde{\mathbf{x}}(t) +  \mathbf{u}(t) + \dot{\bm{\phi}}(t) \\
&\qquad \qquad  + \gamma \bm{\phi}(t) \bigg] + \sum^{r\ell}_{j=1} \kappa_j(t)\dot{\kappa}_j(t) - \kappa^* \sum^{r\ell}_{j=1} \dot{\kappa}_j(t).
\end{align*}
Since $B\left(B^TB\right)^+B^T = M$ and $B \left(B^TB\right)^+ B^TB = B$, substituting \eqref{Eqn:u1} and \eqref{Eqn:kappa1} yields
\begin{align*}
&\dot{V} = -\gamma \tilde{\mathbf{x}}^T(t) \left( M \otimes I_r\right) \tilde{\mathbf{x}}(t) - \kappa^* \left\| \mathbf{y}(t) \right\|_1  \\
&\quad  + {\mathbf{y}}^T(t) \left( \left(B^TB\right)^+ \otimes I_r\right) \left( B^T \otimes I_r\right) \left( \dot{\bm{\phi}}(t) + \gamma \bm{\phi}(t) \right),\\
& \leq -\gamma \tilde{\mathbf{x}}^T(t) \left( M \otimes I_r\right) \tilde{\mathbf{x}}(t) - \kappa^* \left\| \mathbf{y}(t) \right\|_1  + \left\| \mathbf{y}(t) \right\|_1 \\
&  \times \left\| \left( \left(B^TB\right)^+ \otimes I_r\right) \right\|_{\infty} \left\| \left( B^T \otimes I_r\right) \left( \dot{\bm{\phi}}(t) + \gamma \bm{\phi}(t) \right) \right\|_{\infty}.
\end{align*}
Note that $\left\| \left( \left(B^TB\right)^+ \otimes I_r\right) \right\|_{\infty}$ is upper bounded for a connected undirected network under consideration. Also from Assumption~\ref{Assump:Phi}, we have bounded $\left\| \left( B^T \otimes I_r\right) \left( \dot{\bm{\phi}}(t) + \gamma \bm{\phi}(t) \right) \right\|_{\infty}$. Therefore, if $\kappa^*$ is such that
$$\kappa^* \geq \left\| \left( \left(B^TB\right)^+ \otimes I_r\right) \right\|_{\infty} \left\| \left( B^T \otimes I_r\right) \left( \dot{\bm{\phi}}(t) + \gamma \bm{\phi}(t) \right) \right\|_{\infty}$$
we have
\[
\dot{V} \leq -\gamma {\mathbf{y}}^T(t) \left( \left(B^TB\right)^+ \otimes I_r\right) {\mathbf{y}}(t).
\]
\red{
where we used the fact that
\begin{align*}
{\mathbf{y}}^T(t) \left( \left(B^TB\right)^+ \otimes I_r\right) {\mathbf{y}}(t) &=\tilde{\mathbf{x}}^T(t) \left( M \otimes I_r\right) \tilde{\mathbf{x}}(t).
\end{align*}
}
%
Thus $V$ is upper bounded and therefore ${\mathbf{y}}(t)$ and $K(t)$ are bounded. Because of Assumption~\ref{Assump:Phi}, boundedness of ${\mathbf{y}}(t)$ and $K(t)$ implies bounded $\dot{{\mathbf{y}}}(t)$. Since $V$ is lower bounded \red{by} zero and $\dot{V} \leq  - \gamma \sigma_{\min}^+\left( \left(B^TB\right)^+ \right) \|{\mathbf{y}}(t)\|_2$, \red{where $\sigma_{\min}^+\left( \cdot \right)$ denotes the minimum non-zero singular value},
we have $\displaystyle \int_{t_0}^{\infty} \left({\mathbf{y}}^T(t){\mathbf{y}}(t)\right)^{1/2} dt < \infty$, i.e., ${\mathbf{y}}(t)$ is square-integrable. Now based on the BarB\u{a}lat's Lemma (Lemma 3.2.5~\cite{ioannou2013robust}), we have $\displaystyle \lim_{t\rightarrow\infty} {\mathbf{y}}(t) = \mathbf{0}_{r\ell}$. This completes the proof\red{.}
\end{proof}

\red{
\begin{remark}
Note that \red{the solution} to
\begin{align}\label{Eq:Filippov}
\begin{split}
  \dot{\mathbf{y}}(t) &=  \left(B^T \otimes I_r \right) \left[- \gamma \tilde{\mathbf{x}}(t) + \dot{\bm{\phi}}(t) + \gamma \bm{\phi}(t) \right]  \\ & \qquad \qquad  - \left(B^TB \otimes I_r\right)K(t) \text{sgn}\left\{ {\mathbf{y}}(t)\right\} ,
\end{split}
\end{align}
is understood in the Filippov sense~\cite{Filippov}. Define a vector field $\bm{\mathfrak{f}}\left(t,{\mathbf{y}}(t)\right):$ $\mathbb{R}\times\mathbb{R}^{r\ell} \mapsto \mathbb{R}^{r\ell}$ $\triangleq$ $\left(B^T \otimes I_r \right) \left[- \gamma \tilde{\mathbf{x}}(t) + \dot{\bm{\phi}}(t) + \gamma \bm{\phi}(t) \right]$ $-$ $\left(B^TB \otimes I_r\right)$ $K(t)$ $\text{sgn}\left\{ {\mathbf{y}}(t) \right\}$. Note that the Filippov set-valued map for the vector field $\bm{\mathfrak{f}}\left(t,{\mathbf{y}}(t)\right)$ is multiple-valued only at the point of discontinuity, i.e., at the origin. Therefore, the aforementioned stability analysis using a smooth Lyapunov function is valid because the function $V$ is decreasing along every Filippov solution of \eqref{Eq:Filippov} that starts on $\mathbb{R}^{r\ell}\backslash\{\mathbf{0}\}$. Thus, ${\mathbf{y}}(t)$ is globally asymptotically stable.
\end{remark}
}

\begin{remark}
Note that the robust dynamic average consensus algorithm in \eqref{RDATsys1} only requires the existence of the upper bound $\varphi$ and $\dot{\varphi}$. These bounds are not needed for the implementation of the algorithm.
\end{remark}

%

\red{\subsection{Implementation}}

Even though the vector notation used in previous section makes the analysis of the algorithm much easier, it fails to provide the intuition required for distributed implementation. Therefore, here we discuss the distributed implementation of the robust dynamic average consensus algorithm.

After substituting the control~\eqref{Eqn:u1} the robust dynamic average consensus algorithm \eqref{RDATsys1} can be written as
\begin{align}\label{RDATsys2}
  \dot{\mathbf{z}}(t) &=  -\gamma\mathbf{z}(t) - \left(B \otimes I_r\right)K(t)\text{sgn}\left\{\left(B^T \otimes I_r\right) \mathbf{x}(t) \right\}.
\end{align}
Here $K(t)$ can be considered as a pseudo edge-weights multiplying the terms\footnote{In order to simplify the analysis, readers may assume $r=1$.} $\text{sgn}\{ x_i(t) - x_l(t) \}$, $i,l\,\in \mathcal{I}$. In order to compute the term $\left(B \otimes I_r\right)K(t)\text{sgn}\left\{\left(B^T \otimes I_r\right) \mathbf{x}(t) \right\}$ in a distributed manner, agents need to either coordinate among their neighbors to make sure that the gain multiplying $\text{sgn}\{ x_i(t) - x_l(t) \}$ is the same as the gain multiplying $\text{sgn}\{ x_l(t) - x_i(t) \}$ or  constantly exchange their link gain $\kappa_j(t)$ along with their state $x_i(t)$. Due to the nature of adaptive law \eqref{Eqn:kappa1}, agents can easily coordinate their link gains by simply exchanging the initial gains $\kappa_j(t_0)$. Note that when $r=1$, there is a single scaler gain $\kappa_j(t)$ associated with the link $e_j \in \mathcal{E}$, for all $j\in\{1,\ldots,\ell\}$. If the link $e_j$ is between nodes $v_i$ and $v_l$, i,e., $e_j = (v_i,v_l)$, then we can use the notation $\mu_{i,l}(t)$ to denote $\kappa_j(t)$. For $r>1$, $\mu_{i,l}(t)$ is an $r$ dimensional vector with the adaptive law
\begin{equation}\label{Eqn:kappa2}
\dot{\mu}_{i,l}(t) = \left| x_i(t) - x_l(t) \right|, \quad \forall\,i,\,l\,:(v_i,v_l)\in\mathcal{E}.
\end{equation}
Let $\mu_{i,l}(t) =0$ and $\dot{\mu}_{i,l}(t) =0$ for all $t \geq t_0$ and $i,\,l\,:(v_i,v_l)\notin\mathcal{E}$. Note that the agents coordinate the initial condition $\mu_{i,l}(t)$ to \red{ensure that} $\mu_{i,l}(t) = \mu_{l,i}(t)$ for all $t\geq t_0$. Now \eqref{RDATsys2} can be written as $\forall i\in\mathcal{I}$
\begin{align}\label{RDATsys3a}
  \dot{z}_i(t) &=  -\gamma{z}_i(t) - 2 \sum^{n}_{j=1} \, \text{diag}\left[ \mu_{i,j}(t) \right] \text{sgn}\{ x_i(t) - x_j(t) \},
\end{align}
where $\text{diag}\left[ \mu_{i,j}(t) \right]$ is an $r\times r$ diagonal matrix with $\mu_{i,j}(t)$ as its diagonal entries. The constant $2$ is a byproduct of the way in which we defined $B$, i.e., each undirected edge is considered as two distinct directed edges. Each agent computes $x_i(t)$ as
\begin{align}\label{RDATsys3b}
  x_i(t) &=  {z}_i(t) + \phi_i(t).
\end{align}

\section{Dynamic Event-Triggered Robust Dynamic Average Consensus Algorithm}\label{sec:ETRobustDAC}


In order to implement the distributed algorithms \eqref{RDATsys3a} and \eqref{Eqn:kappa2}, every agent $v_i\in\mathcal{V}$ has to know the continuous-time state $x_j(t) = z_j(t) + \phi_j(t)$; $\forall v_j\in{\mathcal{N}_i}$. In other words, continuous communication between agents is needed. However, distributed networks are normally resources-constrained and communication is energy consuming.
In order to avoid continuous communication, here we propose an event-triggered version of the robust dynamic average consensus algorithm.

\subsection{Dynamic Event-Triggered Algorithm}

Inspired by the idea of event-triggered control for multi-agent systems \cite{Dimarogonas12}, we consider the following event-triggered versions of the robust dynamic average consensus algorithm and the adaptive law given in \eqref{RDATsys3a} and \eqref{Eqn:kappa2}, respectively:
\begin{align}
\begin{split}\label{RDATsys4a}
    \dot{z}_i(t) &=  -\gamma{z}_i(t) - 2 \sum^{n}_{j=1} \, \text{diag}\left[ \mu_{i,j}(t) \right] \text{sgn}\{ \hat{x}_i(t) - \hat{x}_j(t) \},
\end{split}\\
    \dot{\mu}_{i,j}(t) &= \left| \hat{x}_i(t) - \hat{x}_j(t) \right| \label{RDATsys4b},
\end{align}
where $\hat{x}_i(t) = x_i(t^i_{k_i(t)})$ denotes the last broadcasted estimate $x_i(t)$ of agent $v_i$ and $t^i_{k_i(t)} = \max\,\{ t^i_k: t^i_k \leq t \}$ and $t^i_0,\,t^i_1,\,\ldots,\,$ is the sequence of event times of agent $v_i$. Note that during inter-event time, the signals $\hat{x}_i(t)$ are held constant for all $i\in\mathcal{I}$.


Define $\hat{\mathbf{x}}(t) \triangleq \begin{bmatrix} \hat{x}^T_1(t) & \ldots & \hat{x}^T_n(t)\end{bmatrix}^T$ and $\hat{\bm{\xi}}(t) \triangleq  \begin{bmatrix} \hat{\xi}^T_1(t) & \ldots & \hat{\xi}^T_{\ell}(t)\end{bmatrix}^T$ $=$ $\left( B^T \otimes I_r \right)\hat{\mathbf{x}}(t)$, where $\hat{x}_i(t) \in \mathbb{R}^r$ for all $i\in\mathcal{I}$ and $\hat{\xi}_j(t) \in \mathbb{R}^r$ for all $j\in\{ i,\ldots,\ell\}$. Let $\widehat{K}(t)$ denotes the adaptive gain obtained from \eqref{RDATsys4b}. Now following \eqref{DACError}, the dynamic average consensus error can be written in the following compact form:
\begin{align}\label{Eqn:DAC_ET_error}
\begin{split}
\dot{\tilde{\mathbf{x}}}(t) = -\gamma \tilde{\mathbf{x}}(t) &- \left(B \otimes I_r \right) \widehat{K}(t)\text{sgn}\left\{ \hat{\bm{\xi}}(t) \right\} \\&+ \left(M\otimes I_r\right) \left( \dot{\bm{\phi}}(t) + \gamma \bm{\phi}(t) \right).
\end{split}
\end{align}
Define $\bm{w}(t) \in \mathbb{R}^{rn} \triangleq \begin{bmatrix} {w}^T_1(t) & \ldots & {w}^T_n(t)\end{bmatrix}^T = \left(B \otimes I_r\right)\,\widehat{K}(t) \, \text{sgn}\left\{ \hat{\bm{\xi}}(t) \right\}$. Now we have
\begin{align}\label{Eqn:DAC_ET_error1}
\begin{split}
\dot{\tilde{\mathbf{x}}}(t) = -\gamma \tilde{\mathbf{x}}(t) &- \bm{w}(t) + \left(M\otimes I_r\right) \left( \dot{\bm{\phi}}(t) + \gamma \bm{\phi}(t) \right).
\end{split}
\end{align}

\red{\subsection{Convergence Result}}

Before we present the asymptotic convergence of the event-triggered algorithm, we make the following assumption:
\begin{assumption}\label{Assump:beta}
Each agent $v_i$ has a local gain $\beta_i$ such that $\forall\,i\in\mathcal{I}$
\begin{equation}\label{Eqn:beta}
  \beta_i \geq \left( \gamma\varphi + \dot{\varphi} \right) \| \left( B\left(B^TB\right)^+ \otimes I_r\right) \|_{\infty}.
\end{equation}
\end{assumption}
Let $\bm{\varepsilon}(t) \triangleq \begin{bmatrix} \varepsilon^T_1(t) & \ldots & \varepsilon^T_n(t)\end{bmatrix}^T = {\mathbf{x}}(t) - \hat{\mathbf{x}}(t)$. Motivated by \cite{Girard15} and \cite{Xinlei}, we introduce the following internal dynamics to facilitate the design of dynamic triggering mechanism:
\begin{align}\label{Eqn:eta}
\dot{\eta}_i(t) = -\alpha_i\,\eta_i(t) - \delta_i \left( \beta_i\mathbf{1}^T_r | \varepsilon_i(t) | - w_i^T(t)\varepsilon_i(t) \right),\,\,i\in\mathcal{I},
\end{align}
where $\eta_i(t_0) > 0$, $\alpha_i >0$ and $\delta_i \geq 1$ are design parameters and can be arbitrarily chosen.
The internal variables $\eta_i(t)$ are incorporated into the triggering law as shown next.


\begin{theorem}\label{Theorem02}
Given Assumptions~\ref{Assump:Graph},~\ref{Assump:Phi}~and~\ref{Assump:beta}, the event-triggered robust dynamic average consensus algorithm in \eqref{RDATsys4a} and the adaptive law \eqref{RDATsys4b} guarantee that the average consensus error, $\tilde{\mathbf{x}}(t)$, asymptotically decays to zero for any initial condition $\mathbf{z}_0$, if $\,\forall\,i\in\mathcal{I}$, the triggering times $\left\{ t^i_k\right\}^{\infty}_{k=1}$ are determined as $t^i_1 = t_0$ and
\begin{align}\label{ETcond}
\begin{split}
t^i_{k+1} &= \min\big\{ t: \\ &\theta_i \left( \beta_i\mathbf{1}^T_r | \varepsilon_i(t) | - w_i^T(t)\varepsilon_i(t) \right) \geq \eta_i(t), \, t\geq t^i_k \big\},
 \end{split}
\end{align}
where $\theta_i \in (0,1)$ is a positive scalar design parameter, $\beta_i$ is defined in Assumption~\ref{Assump:beta}, and $\eta_i(t)$ is from \eqref{Eqn:eta}.
\end{theorem}

\begin{proof}
Define ${\bm{\xi}}(t)$ $=$ $\left( B^T \otimes I_r \right){\mathbf{x}}(t)$. Since $\left( B^T \otimes I_r \right){\mathbf{x}}(t)$ $=$ $\left( B^T \otimes I_r \right)\tilde{\mathbf{x}}(t)$, from \eqref{Eqn:DAC_ET_error} we have
\begin{align}\label{Eqn:DAC_ET_error3}
\begin{split}
\dot{{\bm{\xi}}}(t) = -\gamma {{\bm{\xi}}}(t) &- \left(B^TB \otimes I_r \right) \widehat{K}(t)\text{sgn}\left\{ \hat{\bm{\xi}}(t) \right\} \\&+ \left(B^T\otimes I_r\right) \left( \dot{\bm{\phi}}(t) + \gamma \bm{\phi}(t) \right).
\end{split}
\end{align}
Now consider the following function
\begin{align}\label{Eqn:V}
\begin{split}
V &= \frac{1}{2} {\bm{\xi}}^T(t) \left( \left(B^TB\right)^+ \otimes I_r\right) {\bm{\xi}}(t) \\&\quad + \frac{1}{2} \sum_{i=1}^{n} \sum_{j=1}^{n}  \left( \mu_{i,j}(t) - \mu^* \right)^T\left( \mu_{i,j}(t) - \mu^* \right),
\end{split}
\end{align}
where $\mu^* \in\mathbb{R}^r$ is to be determined. Now taking the time derivative of $V$ along \eqref{Eqn:DAC_ET_error3} yields
\begin{align*}
&\dot{V} = {\mathbf{x}}^T(t) \left( M \otimes I_r\right) \bigg[-\gamma \tilde{\mathbf{x}}(t) + \dot{\bm{\phi}}(t) + \gamma \bm{\phi}(t)\\
&- \left(B \otimes I_r \right) \widehat{K}(t)\text{sgn}\left\{ \hat{\bm{\xi}}(t) \right\}  \bigg] + \sum_{i,j=1}^{n}  \left( \mu_{i,j}(t) - \mu^* \right)^T \dot{\mu}_{i,j}(t) .
\end{align*}
Note that
\begin{align*}
&{\mathbf{x}}^T(t) \left(B \otimes I_r \right) \widehat{K}(t)\text{sgn}\left\{ \hat{\bm{\xi}}(t) \right\}  = \bm{\varepsilon}^T(t)\left(B \otimes I_r \right) \widehat{K}(t)\\
&\qquad \qquad \times \text{sgn}\left\{ \hat{\bm{\xi}}(t) \right\}  + \sum_{i,j=1}^{n}  \mu_{i,j}^T(t) \left| \hat{x}_i(t) - \hat{x}_j(t) \right|,
\end{align*}
Thus we have
\begin{align*}
&\dot{V} = -\gamma \tilde{\mathbf{x}}^T(t) \left( M \otimes I_r\right) \tilde{\mathbf{x}}(t) + {\mathbf{x}}^T(t) \left( M \otimes I_r\right) \bigg( \dot{\bm{\phi}}(t) \\
& \qquad + \gamma \bm{\phi}(t) \bigg) - \bm{\varepsilon}^T(t)\left(B \otimes I_r \right) \widehat{K}(t) \text{sgn}\left\{ \hat{\bm{\xi}}(t) \right\}  \\
& \qquad - \sum_{i,j=1}^{n}  \left( \mu^* \right)^T  \left| \hat{x}_i(t) - \hat{x}_j(t) \right|.
\end{align*}
Without loss of generality, we let $\mu^* = \bar{\mu} \mathbf{1}_r$, where $\bar{\mu}$ is a positive constant \red{to be determined}. Thus we have
\begin{align*}
\sum_{i,j=1}^{n}  \left( \mu^* \right)^T  \left| \hat{x}_i(t) - \hat{x}_j(t) \right| &= \bar{\mu} \sum_{i,j=1}^{n}  \mathbf{1}_r^T  \left| \hat{x}_i(t) - \hat{x}_j(t) \right|,\\
&= \bar{\mu} \left\| \hat{\bm{\xi}}(t) \right\|_1.
\end{align*}
Note
\begin{align*}
\bm{\varepsilon}^T(t)\left(B \otimes I_r \right) \widehat{K}(t) \text{sgn}\left\{ \hat{\bm{\xi}}(t) \right\} = \sum_{i=1}^{n} \varepsilon_i^T(t)w_i(t)
\end{align*}
and
\begin{align*}
& {\mathbf{x}}^T(t) \left( M \otimes I_r\right) \left( \dot{\bm{\phi}}(t) + \gamma \bm{\phi}(t) \right) = \bm{\varepsilon}^T(t) \left( M \otimes I_r\right) \bigg( \dot{\bm{\phi}}(t) \\
&+ \gamma \bm{\phi}(t) \bigg) + \hat{\bm{\xi}}^T(t) \left( \left(B^TB\right)^+ B^T \otimes I_r\right) \left( \dot{\bm{\phi}}(t) + \gamma \bm{\phi}(t) \right)\\
&\leq \| \bm{\varepsilon}(t) \|_1 \| \left( B\left(B^TB\right)^+ \otimes I_r\right) \|_{\infty} \\
& \times \| \left( B^T \otimes I_r\right) \bigg( \dot{\bm{\phi}}(t) + \gamma \bm{\phi}(t) \bigg) \|_{\infty} + \| \hat{\bm{\xi}}(t) \|_1  \\
& \times \| \left( \left(B^TB\right)^+ \otimes I_r\right) \|_{\infty} \| \left( B^T \otimes I_r\right) \bigg( \dot{\bm{\phi}}(t) + \gamma \bm{\phi}(t) \bigg) \|_{\infty}\\
&\leq \| \bm{\varepsilon}(t) \|_1 \| \left( B\left(B^TB\right)^+ \otimes I_r\right) \|_{\infty} \left( \gamma\varphi + \dot{\varphi} \right)\\
& \qquad + \| \hat{\bm{\xi}}(t) \|_1  \| \left( \left(B^TB\right)^+ \otimes I_r\right) \|_{\infty} \left( \gamma\varphi + \dot{\varphi} \right),
\end{align*}
where the second inequality follows from Assumption \ref{Assump:Phi}. Now an upper bound on $\dot{V}$ can be obtained as
\begin{align*}
&\dot{V} \leq  -\gamma \tilde{\mathbf{x}}^T(t) \left( M \otimes I_r\right) \tilde{\mathbf{x}}(t) - \sum_{i=1}^{n} \varepsilon_i^T(t)w_i(t) \\
&\quad+ \| \bm{\varepsilon}(t) \|_1 \| \left( B\left(B^TB\right)^+ \otimes I_r\right) \|_{\infty} \left( \gamma\varphi + \dot{\varphi} \right)  \\
& \quad + \| \hat{\bm{\xi}}(t) \|_1  \| \left( \left(B^TB\right)^+ \otimes I_r\right) \|_{\infty} \left( \gamma\varphi + \dot{\varphi} \right) - \bar{\mu} \left\| \hat{\bm{\xi}}(t) \right\|_1.
\end{align*}
If $\bar{\mu}$ is selected such that $$\bar{\mu}\geq\left( \gamma\varphi + \dot{\varphi} \right)\| \left( \left(B^TB\right)^+ \otimes I_r\right) \|_{\infty},$$ then we have
\begin{align}\label{Eqn:Vdot_bound}
\begin{split}
\dot{V} \leq  -\gamma \tilde{\mathbf{x}}^T(t) \left( M \otimes I_r\right) \tilde{\mathbf{x}}(t) &- \sum_{i=1}^{n} w_i^T(t) \varepsilon_i(t) \\
&+  \sum_{i=1}^{n} \beta_i \, \mathbf{1}_r^T \left|  \varepsilon_i(t) \right|,
\end{split}
\end{align}
where $\beta_i$ is from Assumption~\ref{Assump:beta}. Now consider a Lyapunov function candidate as follows
\begin{align}
W = V + \sum_{i=1}^{n} {\eta}_i(t),
\end{align}
where $V$ is given in \eqref{Eqn:V}. Thus from \eqref{Eqn:Vdot_bound} and \eqref{Eqn:eta} we have,
\begin{align*}
&\dot{W} = \dot{V} + \sum_{i=1}^{n} \dot{\eta}_i(t)\\
&\leq -\gamma \tilde{\mathbf{x}}^T(t) \left( M \otimes I_r\right) \tilde{\mathbf{x}}(t) +  \sum_{i=1}^{n} \left( \beta_i \, \mathbf{1}_r^T \left|  \varepsilon_i(t) \right| -w_i^T(t) \varepsilon_i(t) \right) \\&-\sum_{i=1}^{n} \alpha_i\,\eta_i(t) - \sum_{i=1}^{n} \delta_i \left( \beta_i\mathbf{1}^T_r | \varepsilon_i(t) | - w_i^T(t)\varepsilon_i(t) \right)\\
&\leq -\gamma \tilde{\mathbf{x}}^T(t) \left( M \otimes I_r\right) \tilde{\mathbf{x}}(t)-\sum_{i=1}^{n} \alpha_i\,\eta_i(t) \leq 0.
\end{align*}
Thus $W$ is upper bounded and therefore ${\bm{\xi}}(t)$, $\widehat{K}(t)$, and $\bm{\eta}(t)$ are all bounded. Because of Assumption~\ref{Assump:Phi}, boundedness of ${\bm{\xi}}(t)$ and $\widehat{K}(t)$ implies bounded $\dot{\bm{\xi}}(t)$. Since $W$ is lower bounded \red{by} zero and $\dot{W} \leq  -\gamma \tilde{\mathbf{x}}^T(t) \left( M \otimes I_r\right) \tilde{\mathbf{x}}(t)$, we have ${\bm{\xi}}(t)$ is square-integrable. Now based on the BarB\u{a}lat's Lemma (Lemma 3.2.5~\cite{ioannou2013robust}), we have $\displaystyle \lim_{t\rightarrow\infty} {\bm{\xi}}(t) = \mathbf{0}_{r\ell}$. Thus agents asymptotically reach consensus on the dynamic average consensus error $\tilde{\mathbf{x}}(t)$. From \eqref{Eqn:DAC_ET_error} we have $\sum\limits_{i=1}^{n} \tilde{x}_i(t)$ is exponentially decreasing to $\mathbf{0}_r$ at the rate $\gamma$. Thus the agents asymptotically reach consensus on $\tilde{x}_i(t)$s and $\tilde{x}_i(t)$s satisfy the zero-sum condition, therefore $\lim_{t\rightarrow\infty} \tilde{x}_i(t) = \mathbf{0}_r$ for all $i\in\mathcal{I}$.
\end{proof}

\red{
\subsection{Exclusion of Zeno Behavior}}

Here we prove that there is no Zeno behavior by contradiction.
But first, based on the Lyapunov analysis given in the proof of Theorem~\ref{Theorem02}, we can conclude that for all $i\in \mathcal{I}$ there exists positive bounds $x^i_{\max}$ and $w^i_{\max}$ such that
\begin{align*}
 \displaystyle \sup_{t\in[0,\infty)} \, \| x_i(t)\|_{\infty} &\leq x_{\max}^i \\
  \displaystyle \sup_{t\in[0,\infty)} \, \| w_i(t)\|_{\infty} &\leq w_{\max}^i.
\end{align*}
Thus based on \eqref{Eqn:DAC_ET_error1} and Assumption \ref{Assump:Phi}, we have
\begin{align*}
 \displaystyle \sup_{t\in[0,\infty)} \, \| \dot{x}_i(t)\|_{\infty} &\leq \dot{x}_{\max}^i.
\end{align*}
Before we present the main result of this section, note that from the triggering mechanism \eqref{ETcond}, we have
\begin{align}
&\theta_i \left( \beta_i\mathbf{1}^T_r | \varepsilon_i(t) | - w_i^T(t)\varepsilon_i(t) \right) \leq \eta_i(t), \quad \forall\,t\geq t_0.
\end{align}
Thus,
\begin{align}
\dot{\eta}_i(t) \geq -\alpha_i\,\eta_i(t) - \frac{\delta_i}{\theta_i} \eta_i(t), \quad \forall\,t\geq t_0.
\end{align}
Therefore,
\begin{align} \label{EtaBound}
{\eta}_i(t) \geq \eta_i(t_0) e ^{-\left( \alpha_i + \frac{\delta_i}{\theta_i}  \right)t} > 0, \quad \forall\,t\geq t_0.
\end{align}

\begin{theorem}\label{Theorem03}
Given Assumptions~\ref{Assump:Phi} and ~\ref{Assump:beta} hold, for any connected undirected network running the event-triggered robust dynamic average consensus algorithm in \eqref{RDATsys4a} and the adaptive law \eqref{RDATsys4b}, the dynamic triggering law \eqref{ETcond} guarantees the exclusion of Zeno behavior.
\end{theorem}

\begin{proof}
Suppose there exists Zeno behavior. Then there exists an agent $v_i$, such that $\displaystyle \lim_{k\rightarrow\infty} \, t^i_k = \sum_{k=0}^{\infty}\,\left( t^i_{k+1} - t^i_k \right)  = T_0$, where $T_0$ is a positive constant. It follows from the existence of the limit that for a constant $\epsilon_0 > 0$, there exists a positive integer $s(\epsilon_0)$ such that
\begin{align}
t^i_k \in \left[ T_0-\epsilon_0, T_0\right], \quad \forall \, k\geq s(\epsilon_0).
\end{align}
Let $\tau$ denote the event-time $t^i_{s(\epsilon_0)+1}$. Note that $\left \| \varepsilon_i \left(\tau^+\right) \right \|^2 = 0$ immediately after the triggering. Just before the triggering, from \eqref{ETcond} we have
\begin{align*}
  \left(  \beta_i\mathbf{1}^T_r | \varepsilon_i(\tau^-) | - \varepsilon_i^T(\tau^-)w_i(\tau^-) \right) &\geq \frac{1}{\theta_i}\eta_i(\tau^-)
\end{align*}
Thus
\begin{align}\label{InEq1}
  \beta_i \| \varepsilon_i(\tau^-) \|_1 - \varepsilon_i^T(\tau^-)w_i(\tau^-)   &\geq \frac{1}{\theta_i}\eta_i(\tau^-)
\end{align}
Note
\begin{align}\label{InEq2}
\begin{split}
  &\beta_i \| \varepsilon_i(\tau^-) \|_1 - \varepsilon_i^T(\tau^-)w_i(\tau^-)   \\
  &\qquad \leq \|\varepsilon_i(\tau^-)\|_2 \|w_i(\tau^-)\|_2  + \beta_i \sqrt{r} \| \varepsilon_i(\tau^-) \|_2\\
  &\qquad \leq  \left( w^i_{\max} + \beta_i \right)\sqrt{r} \| \varepsilon_i(\tau^-) \|_2.
  \end{split}
\end{align}
From \eqref{InEq1} and \eqref{InEq2}, we have
\begin{align}\label{InEq3}
  \| \varepsilon_i(\tau^-) \|_2   &\geq \frac{1}{\theta_i \left( w^i_{\max} + \beta_i \right)\sqrt{r} } \eta_i(\tau^-)
\end{align}
Now from \eqref{EtaBound}, we can conclude that there exists a $t^* \in  \left(t^i_{s(\epsilon_0)},\,\,  t^i_{s(\epsilon_0)+1}\right)$ such that
\begin{align}\label{InEq4}
 \left \| \varepsilon_i \left(t\right) \right \|_2 \geq  \frac{\eta_i(t_0)e ^{-\left( \alpha_i + \frac{\delta_i}{\theta_i}  \right)t}}{\theta_i \left( w^i_{\max} + \beta_i \right)\sqrt{r} }, \,\, \forall \, t \in \left[t^*, t^i_{s(\epsilon_0)+1}\right).
\end{align}
Since $\left \| \varepsilon_i \left( {t^i_{s(\epsilon_0)}}^+ \right) \right \|_2 = 0$ and $\dot{x}_i(t)$ is upper-bounded, we have
\begin{align}\label{InEq5}
\begin{split}
 &\left \| \varepsilon_i \left( t \right) \right \|_2 = \| \hat{x}_i(t) - x_i(t) \|_2 \\
 &\quad \leq \left(t - t^i_{s(\epsilon_0)}\right) \sqrt{r} \dot{x}^i_{\max}, \,\, \forall \, t \in \left[t^*, t^i_{s(\epsilon_0)+1}\right).
\end{split}
\end{align}
Combining \eqref{InEq4} and \eqref{InEq5} yields
\begin{align}
\begin{split}
\left(t - t^i_{s(\epsilon_0)}\right) \sqrt{r} \dot{x}^i_{\max} \geq \frac{\eta_i(t_0)e ^{-\left( \alpha_i + \frac{\delta_i}{\theta_i}  \right)t}}{\theta_i \left( w^i_{\max} + \beta_i \right)\sqrt{r} },\qquad&\\ \qquad \forall \, t \in \left[t^*, t^i_{s(\epsilon_0)+1}\right).&
\end{split}
\end{align}
Thus we have
\begin{align*}
\left(t - t^i_{s(\epsilon_0)}\right)  \geq \frac{\eta_i(t_0)e ^{-\left( \alpha_i + \frac{\delta_i}{\theta_i}  \right)t}}{\theta_i \left( w^i_{\max} + \beta_i \right){r} \dot{x}^i_{\max} },\,\, \forall \, t \in \left[t^*, t^i_{s(\epsilon_0)+1}\right).&
\end{align*}
Clearly, $\left(t^i_{s(\epsilon_0)+1} - t^i_{s(\epsilon_0)}\right) \geq \left(t - t^i_{s(\epsilon_0)}\right) $ and thus
\begin{align*}
\left(t^i_{s(\epsilon_0)+1} - t^i_{s(\epsilon_0)}\right)   &\geq \frac{\eta_i(t_0)}{\theta_i \left( w^i_{\max} + \beta_i \right){r} \dot{x}^i_{\max} } e ^{-\left( \alpha_i + \frac{\delta_i}{\theta_i}  \right)t},\\
& \qquad \qquad \qquad \qquad \forall \, t \in \left[t^*, t^i_{s(\epsilon_0)+1}\right)\\
&\geq \frac{\eta_i(t_0)}{\theta_i \left( w^i_{\max} + \beta_i \right){r} \dot{x}^i_{\max} } e ^{-\left( \alpha_i + \frac{\delta_i}{\theta_i}  \right) T_0}
\end{align*}
where $T_0 = \displaystyle \lim_{k\rightarrow\infty} t^i_k$. Thus the inter-event times are lower-bounded which contradicts the existence of the limit $T_0$. In fact,
$$ \rho = \frac{\eta_i(t_0)}{\theta_i \left( w^i_{\max} + \beta_i \right){r} \dot{x}^i_{\max} } e ^{-\left( \alpha_i + \frac{\delta_i}{\theta_i}  \right) T_0}$$
is the lower-bound on inter-events for $t \geq t^i_{s(\epsilon_0)}$. This clearly contradicts the existence of Zeno behavior.
\end{proof}

\section{Simulation}

Consider an undirected network of $10$ nodes given in Fig.~\ref{graph}. Initially, the network is connected as shown in Fig.~\ref{FigG1} and at $t = 2.5$, link $(v_3,v_7)$ (and thus $(v_7,v_3)$) is severed, resulting in two disconnected graphs for the rest of the simulation time, see Fig.~\ref{FigG2}.
\begin{figure}[!ht]
\centering{
\subfigure[During $0\leq t < 2.5$]{
\begin{tikzpicture}[-,node distance=1cm,
  thick,main node/.style={circle,fill=blue!20,draw,font=\sffamily\tiny\bfseries}]
  \node[main node] (1) {1};
  \node[main node] (2) [below of=1] {2};
  \node[main node] (3) [left of=2] {3};
  \node[main node] (4) [above of=3] {4};
  \node[main node] (5) [above  of=4] {5};
  \node[main node] (7) [below left of=3] {7};
  \node[main node] (6) [ left of=4] {6};
  \node[main node] (9) [below left of=6] {9};
  \node[main node] (8) [below of=9] {8};
 \node[main node] (10) [right of=9] {\scalebox{.7}{10}};
\draw (1) -- (2)
(2) -- (4)
(3) -- (4)
(4) -- (5)
(4) -- (6)
(3) -- (7)
(7) -- (8)
(6) -- (5)
(8) -- (9)
(3) -- (2)
(9) -- (10);
\end{tikzpicture}\label{FigG1}}$\qquad\qquad$
\subfigure[During $2.5 \leq t $]{
\begin{tikzpicture}[-,node distance=1cm,
  thick,main node/.style={circle,fill=blue!20,draw,font=\sffamily\tiny\bfseries}]
  \node[main node] (1) {1};
  \node[main node] (2) [below of=1] {2};
  \node[main node] (3) [left of=2] {3};
  \node[main node] (4) [above of=3] {4};
  \node[main node] (5) [above  of=4] {5};
  \node[main node] (7) [below left of=3] {7};
  \node[main node] (6) [ left of=4] {6};
  \node[main node] (9) [below left of=6] {9};
  \node[main node] (8) [below of=9] {8};
 \node[main node] (10) [right of=9] {\scalebox{.7}{10}};
\draw (1) -- (2)
(2) -- (4)
(3) -- (4)
(4) -- (5)
(4) -- (6)
(7) -- (8)
(6) -- (5)
(8) -- (9)
(3) -- (2)
(9) -- (10);
\end{tikzpicture}\label{FigG2}}
}
\caption{Network.}
\label{graph}
\end{figure}
The individual reference signals are given as
\begin{align*}
  \begin{array}{ll}
    \phi_i(t) = a_i\sin(\omega_i t + \psi_i)),   \quad \forall\,i\in\{1,\ldots,5\}, & \\
    \phi_i(t) = a_i\cos(\omega_i t + \psi_i)),,   \quad \forall\,i\in\{6,\ldots,10\}, & \\
    \quad a_i = \frac{i-1}{2}-7, \quad \omega_i = \frac{1}{4}(i+1),  \quad  \psi_i(t) = \frac{2\pi i}{n} -\pi. &
  \end{array}
\end{align*}
For simulation purposes, we select $\varphi = 10$, $\dot{\varphi} = 20$ and $\gamma = 1$. The individual design parameters are selected as $\alpha_i = 3$, $\beta_i = 100$, $\delta_i = 1.5$, and $\theta_i = 0.9$ for all $i$. Here we use random initial conditions for $z_i(0)$, $\kappa_j(0)$ and $\eta_i(0)$ with the constraints $\kappa_j(0) \geq 10$ and $\eta_i(0) \geq 1$.

In our simulation, the sample length is $10^{-3}$. Under the proposed dynamic event-triggered communication mechanism, during time interval $[0,5]$, agents $1-10$ triggered $26\%$, $25\%$, $31\%$, $44\%$, $25\%$, $29\%$, $31\%$, $38\%$, $33\%$, and $40\%$ of times. Therefore, our dynamic event-triggered communication mechanism is very efficient and avoids inter-agent communication about $55-75\%$ of the time.

\begin{figure}[!ht]
  \begin{centering}
      \subfigure[Individual $\phi_i(t)$ (solid lines) and corresponding $\bar{\phi}(t)$ (thick dotted lines).]{
      \psfrag{Time}[][]{\footnotesize{$t$}}
      \psfrag{phi}[][]{\footnotesize{$\phi(t)$}}
      \includegraphics[width=.23\textwidth]{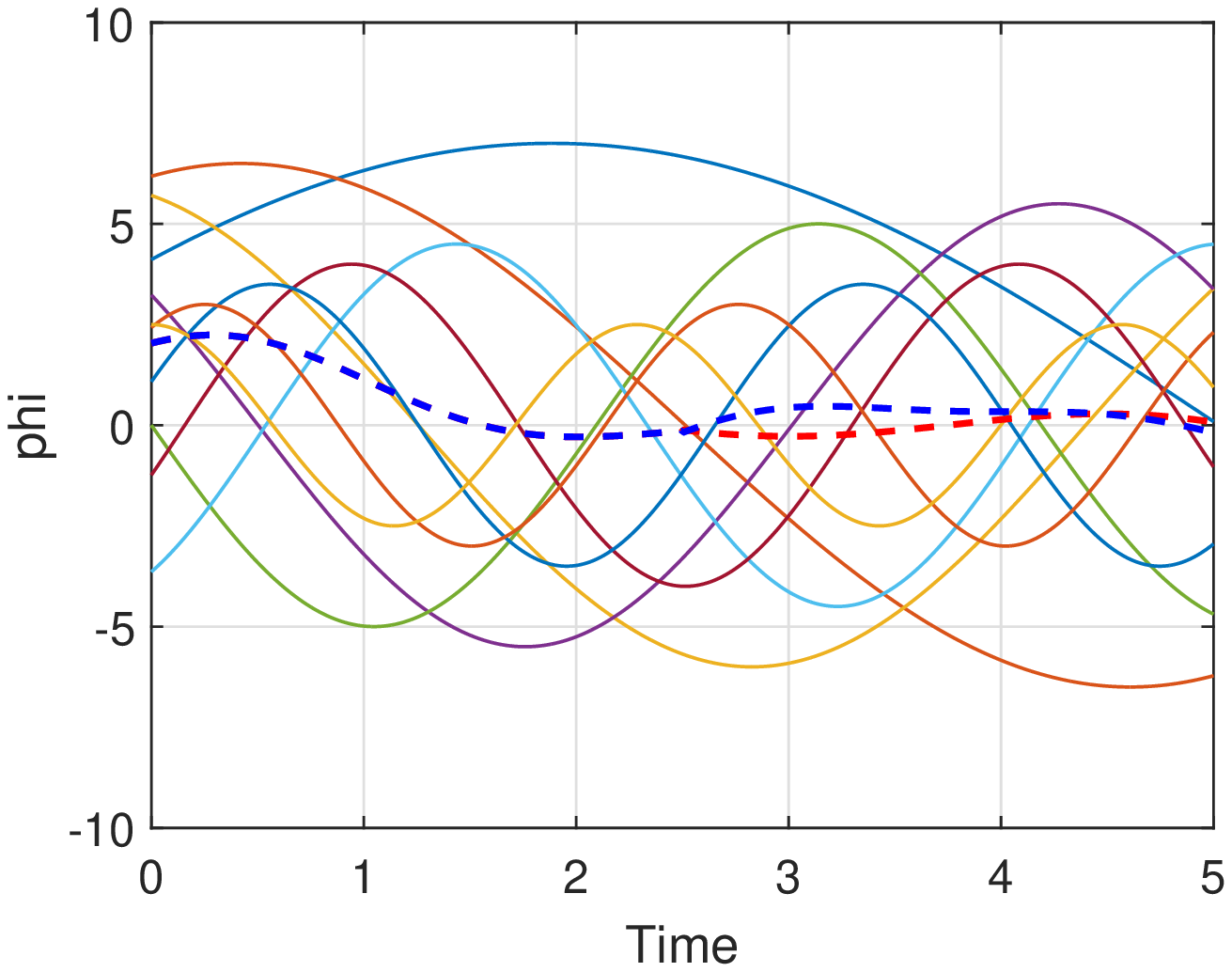}\label{Fig:Phi}}
      \subfigure[Network estimates of $\bar{\phi}(t)$.]{
      \psfrag{Time}[][]{\footnotesize{$t$}}
      \psfrag{p}[][]{\footnotesize{$x(t)$}}
      \psfrag{Avg1}[][]{\tiny{$\bar{\phi}_1$}}
      \psfrag{Avg2}[][]{\tiny{$\bar{\phi}_2$}}
      \psfrag{x1}[][]{\tiny{$x_1$}}
      \psfrag{x2}[][]{\tiny{$x_2$}}
      \psfrag{x3}[][]{\tiny{$x_3$}}
      \psfrag{x4}[][]{\tiny{$x_4$}}
      \psfrag{x5}[][]{\tiny{$x_5$}}
      \psfrag{x6}[][]{\tiny{$x_6$}}
      \psfrag{x7}[][]{\tiny{$x_7$}}
      \psfrag{x8}[][]{\tiny{$x_8$}}
      \psfrag{x9}[][]{\tiny{$x_9$}}
      \psfrag{x10}[][]{\tiny{$x_{10}$}}
      \includegraphics[width=.23\textwidth]{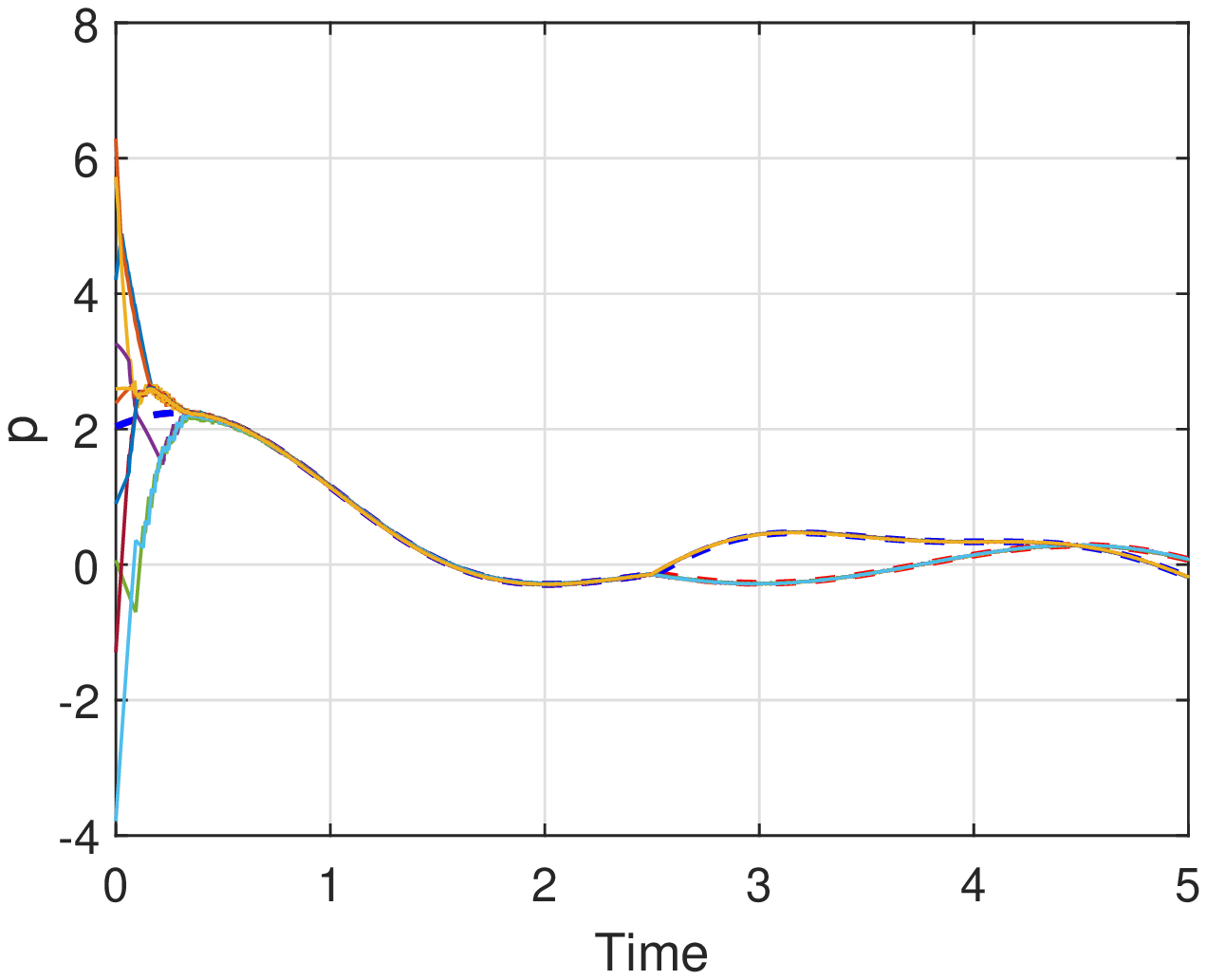}\label{Fig:X}}
      \caption{Individual reference signals and network estimates of the time-varying average.}
      \label{Plot1}
  \end{centering}
\end{figure}

Individual reference signals for each agent and the corresponding average $\bar{\phi}(t)$ are given in Fig. \ref{Fig:Phi}. Note that after the network splits, there are two $\bar{\phi}(t)$s, one corresponding to each of the connected components. The results obtained from implementing the proposed event-triggered dynamic average consensus estimators are given in Fig. \ref{Fig:X}. Note that the proposed event-triggered algorithm is able to achieve accurate estimates of $\bar{\phi}(t)$ even with network interruptions. The estimation error for the proposed dynamic average consensus estimator are given in Figs.~\ref{Fig:Error1} and \ref{Fig:Error2}. Figure \ref{Fig:Error1} contains the estimation error for agents $1-6$ while Fig. \ref{Fig:Error2} contains the estimation error for agents $7-10$.


\begin{figure}[!ht]
  \begin{centering}
      \subfigure[Consensus error for agents $1-6$.]{
      \psfrag{Time}[][]{\footnotesize{$t$}}
      \psfrag{ConsensusError}[][]{\footnotesize{$\tilde{x}(t)$}}
      \includegraphics[width=.23\textwidth]{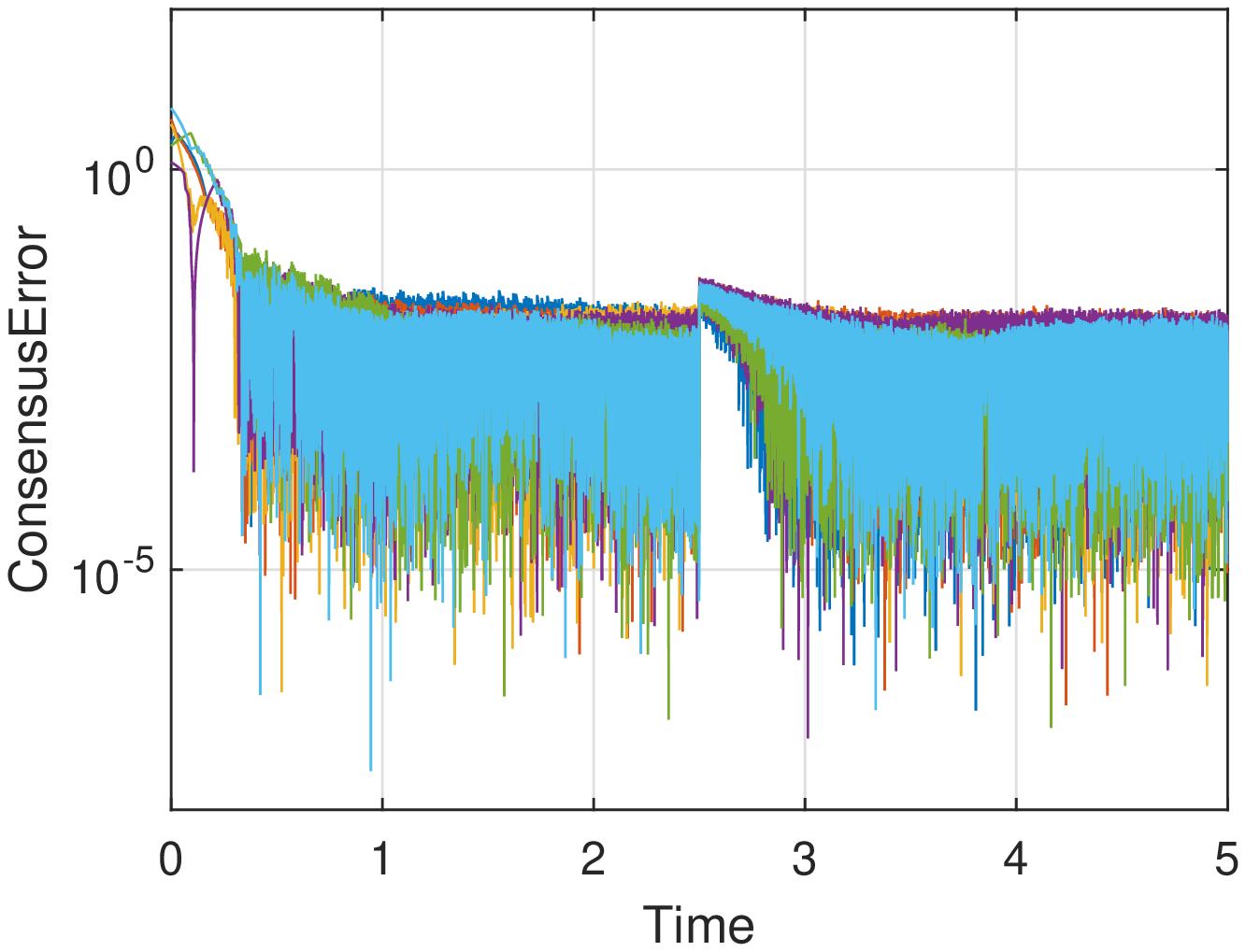}\label{Fig:Error1}}
      \subfigure[Consensus error for agents $7-10$.]{
      \psfrag{Time}[][]{\footnotesize{$t$}}
      \psfrag{ConsensusError}[][]{\footnotesize{$\tilde{x}(t)$}}
      \includegraphics[width=.23\textwidth]{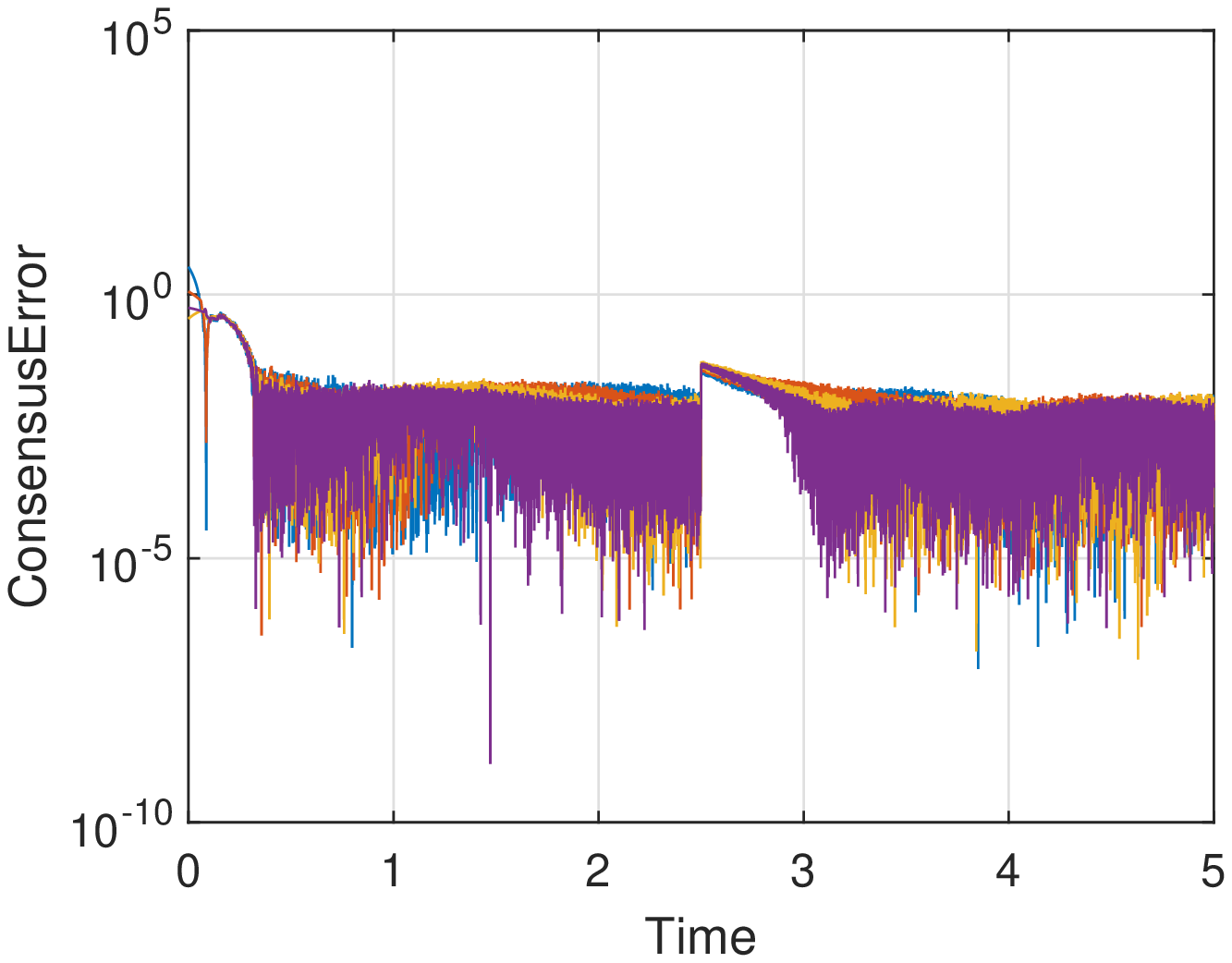}\label{Fig:Error2}}
      \caption{Dynamic average consensus error $\tilde{x}(t)$.}
      \label{Plot1}
  \end{centering}
\end{figure}

\section{Conclusion}\label{sec:conclusion}

Here we present a fully distributed dynamic average consensus algorithm that is robust to initialization errors. The proposed robust algorithm makes use of an adaptive gain which removes the explicit use of any upper bounds on reference signals or its time-derivatives in the algorithm. Since this robust algorithm requires continuous communication among agents, we introduce a novel dynamic event-triggered communication scheme to reduce the overall inter-agent interactions. The proposed triggering laws involve internal dynamic variables which play an essential role in guaranteeing that the triggering time sequence does not exhibit Zeno behavior. Asymptotic convergence of both the continuous communication based algorithm and the event-triggered algorithm are presented. Future work include applying the developed algorithm to distributed learning and control problems as well as extending the current approach by considering directed networks.

\bibliography{Biblio}

\begin{thebibliography}{10}
\providecommand{\url}[1]{#1}
\csname url@samestyle\endcsname
\providecommand{\newblock}{\relax}
\providecommand{\bibinfo}[2]{#2}
\providecommand{\BIBentrySTDinterwordspacing}{\spaceskip=0pt\relax}
\providecommand{\BIBentryALTinterwordstretchfactor}{4}
\providecommand{\BIBentryALTinterwordspacing}{\spaceskip=\fontdimen2\font plus
\BIBentryALTinterwordstretchfactor\fontdimen3\font minus
  \fontdimen4\font\relax}
\providecommand{\BIBforeignlanguage}[2]{{%
\expandafter\ifx\csname l@#1\endcsname\relax
\typeout{** WARNING: IEEEtran.bst: No hyphenation pattern has been}%
\typeout{** loaded for the language `#1'. Using the pattern for}%
\typeout{** the default language instead.}%
\else
\language=\csname l@#1\endcsname
\fi
#2}}
\providecommand{\BIBdecl}{\relax}
\BIBdecl

\bibitem{George}
J.~George, ``Networked sensing and distributed \textsc{K}alman-\textsc{B}ucy
  filtering based on dynamic average consensus,'' in \emph{Proc. IEEE
  International Conference on Distributed Computing in Sensor Systems}, May
  2013, pp. 175 -- 182.

\bibitem{George2018ICASSP}
------, ``Distributed maximum likelihood using dynamic average consensus,'' in
  \emph{Proc. IEEE International Conference on Acoustics, Speech and Signal
  Processing (ICASSP)}, Apr. 2018, pp. 3834--3838.

\bibitem{Spanos2005}
D.~P. Spanos, R.~Olfati-Saber, and R.~M. Murray, ``Dynamic consensus on mobile
  networks,'' in \emph{Proc. 16th IFAC World Congress}, Prague, Czech, 2005.

\bibitem{Freeman:CDC2006}
R.~A. Freeman, P.~Yang, and K.~M. Lynch, ``Stability and convergence properties
  of dynamic average consensus estimators,'' in \emph{Proc. 45th IEEE
  Conference on Decision and Control}, Dec. 2006, pp. 338 -- 343.

\bibitem{Bai10}
H.~Bai, R.~Freeman, and K.~Lynch, ``Robust dynamic average consensus of
  time-varying inputs,'' in \emph{Proc. 49th IEEE Conference on Decision and
  Control}, Dec. 2010, pp. 3104 -- 3109.

\bibitem{Bai2016}
H.~Bai, ``A two-time-scale dynamic average consensus estimator,'' in
  \emph{Proc. 55th Conference on Decision and Control}, Dec. 2016, pp. 75 --
  80.

\bibitem{Kia2013}
S.~S. Kia, J.~Cort\'{e}s, and S.~Mart\'{\i}nez, ``Singularly perturbed
  algorithms for dynamic average consensus,'' in \emph{Proc. European Control
  Conference}, Jul. 2013, pp. 1758 -- 1763.

\bibitem{Nosrati2012}
S.~Nosrati, M.~Shafiee, and M.~B. Menhaj, ``Dynamic average consensus via
  nonlinear protocols,'' \emph{Automatica}, vol.~48, no.~9, pp. 2262 -- 2270,
  2012.

\bibitem{Kia2014}
S.~S. Kia, J.~Cort\'{e}s, and S.~Mart\'{\i}nez, ``Dynamic average consensus
  with distributed event-triggered communication,'' in \emph{Proc. 53rd IEEE
  Conference on Decision and Control}, Dec 2014, pp. 890--895.

\bibitem{Kia2015a}
------, ``Dynamic average consensus under limited control authority and privacy
  requirements,'' \emph{International Journal of Robust and Nonlinear Control},
  vol.~25, no.~13, pp. 1941 -- 1966, 2015.

\bibitem{Zhu2010}
M.~Zhu and S.~Mart\'{\i}nez, ``Discrete-time dynamic average consensus,''
  \emph{Automatica}, vol.~46, no.~2, pp. 322 -- 329, 2010.

\bibitem{Yuan2012}
Y.~Yuan, J.~Liu, R.~M. Murray, and J.~Gon\c{c}alves, ``Decentralised
  minimal-time dynamic consensus,'' in \emph{Proc. American Control
  Conference}, Jun. 2012, pp. 800 -- 805.

\bibitem{Montijano2014b}
E.~Montijano, J.~I. Montijano, C.~Sag\"{u}\'{e}s, and S.~Mart\'{\i}nez, ``Step
  size analysis in discrete-time dynamic average consensus,'' in \emph{Proc.
  American Control Conference}, Jun. 2014, pp. 5127 -- 5132.

\bibitem{Scoy2015}
B.~V. Scoy, R.~A. Freeman, and K.~M. Lynch, ``A fast robust nonlinear dynamic
  average consensus estimator in discrete time,'' in \emph{Proc. 5th
  \textsc{IFAC} Workshop on Distributed Estimation and Control in Networked
  Systems}, Sep. 2015, pp. 191 -- 196.

\bibitem{VanScoy2015CDC}
------, ``Design of robust dynamic average consensus estimators,'' in
  \emph{Proc. 54th IEEE Conference on Decision and Control}, Dec. 2015, pp.
  6269 -- 6275.

\bibitem{Franceschelli2016}
M.~Franceschelli and A.~Gasparri, ``Multi-stage discrete time dynamic average
  consensus,'' in \emph{Proc. 55th Conference on Decision and Control}, Dec.
  2016, pp. 897 -- 903.

\bibitem{WeiRen2011ACC}
F.~Chen, Y.~Cao, and W.~Ren, ``Distributed computation of the average of
  multiple time-varying reference signals,'' in \emph{Proc. American Control
  Conference}, Jun. 2011, pp. 1650 -- 1655.

\bibitem{WeiRen2012TAC}
------, ``Distributed average tracking of multiple time-varying reference
  signals with bounded derivatives,'' \emph{IEEE Transactions on Automatic
  Control}, vol.~57, no.~12, pp. 3169 -- 3174, Dec. 2012.

\bibitem{WeiRen2013CCC}
F.~Chen and W.~Ren, ``Robust distributed average tracking for coupled general
  linear systems,'' in \emph{Proc. 32nd Chinese Control Conference}, Jul. 2013,
  pp. 6953 -- 6958.

\bibitem{WeiRen2014ACC}
F.~Chen, G.~Feng, L.~Liu, and W.~Ren, ``An extended proportional-integral
  control algorithm for distributed average tracking and its applications in
  \textsc{E}uler-\textsc{L}agrange systems,'' in \emph{Proc. American Control
  Conference}, Jun. 2014, pp. 2581 -- 2586.

\bibitem{WeiRen2015ACC}
S.~Ghapani, W.~Ren, and F.~Chen, ``Distributed average tracking for
  double-integrator agents without using velocity measurements,'' in
  \emph{Proc. American Control Conference}, Jul. 2015, pp. 1445 -- 1450.

\bibitem{WeiRen2015TAC2}
F.~Chen, G.~Feng, L.~Liu, and W.~Ren, ``Distributed average tracking of
  networked \textsc{E}uler-\textsc{L}agrange systems,'' \emph{IEEE Transactions
  on Automatic Control}, vol.~60, no.~2, pp. 547 -- 552, Feb. 2015.

\bibitem{WeiRen2015TAC3}
F.~Chen, W.~Ren, W.~Lan, and G.~Chen, ``Distributed average tracking for
  reference signals with bounded accelerations,'' \emph{IEEE Transactions on
  Automatic Control}, vol.~60, no.~3, pp. 863 -- 869, Mar. 2015.

\bibitem{WeiRen2016ACC}
S.~Ghapani, S.~Rahili, and W.~Ren, ``Distributed average tracking for
  second-order agents with nonlinear dynamics,'' in \emph{Proc. American
  Control Conference}, Jul. 2016, pp. 4636 -- 4641.

\bibitem{Zhao2017}
Y.~Zhao, Y.~Liu, Z.~Li, and Z.~Duan, ``Distributed average tracking for
  multiple signals generated by linear dynamical systems: An edge-based
  framework,'' \emph{Automatica}, vol.~75, pp. 158 -- 166, 2017.

\bibitem{WeiRen2017ACC}
S.~Rahili and W.~Ren, ``Heterogeneous distributed average tracking using
  nonsmooth algorithms,'' in \emph{Proc. American Control Conference}, May.
  2017, pp. 691 -- 696.

\bibitem{Kia2015b}
S.~S. Kia, J.~Cort\'{e}s, and S.~Mart\'{\i}nez, ``Distributed event-triggered
  communication for dynamic average consensus in networked systems,''
  \emph{Automatica}, vol.~59, pp. 112 -- 119, 2015.

\bibitem{George2017ACC}
J.~George, R.~A. Freeman, and K.~M. Lynch, ``Robust dynamic average consensus
  algorithm for signals with bounded derivatives,'' in \emph{Proc. American
  Control Conference (ACC)}, May 2017, pp. 352--357.

\bibitem{Johansson99}
K.~H. Johansson, M.~Egerstedt, J.~Lygeros, and S.~Sastry, ``On the
  regularization of zeno hybrid automata,'' \emph{Systems \& Control Letters},
  vol.~38, no.~3, pp. 141 -- 150, 1999.

\bibitem{Gutman04}
W.~X. I.~Gutman, ``Generalized inverse of the \textsc{L}aplacian matrix and
  some applications,'' \emph{Bulletin, Classe des Sciences Math\'{e}matiques et
  Naturelles, Sciences math\'{e}matiques}, vol. 129, no.~29, pp. 15--23, 2004.

\bibitem{ioannou2013robust}
P.~Ioannou and J.~Sun, \emph{Robust Adaptive Control}.\hskip 1em plus 0.5em
  minus 0.4em\relax Dover Publications, 2013.

\bibitem{Filippov}
A.~Filippov, \emph{Differential Equations with Discontinuous Righthand Sides:
  Control Systems}, ser. Mathematics and its Applications.\hskip 1em plus 0.5em
  minus 0.4em\relax Springer Netherlands, 1988.

\bibitem{Dimarogonas12}
D.~V. Dimarogonas, E.~Frazzoli, and K.~H. Johansson, ``Distributed
  event-triggered control for multi-agent systems,'' \emph{IEEE Transactions on
  Automatic Control}, vol.~57, no.~5, pp. 1291--1297, May 2012.

\bibitem{Girard15}
A.~Girard, ``Dynamic triggering mechanisms for event-triggered control,''
  \emph{IEEE Transactions on Automatic Control}, vol.~60, no.~7, pp. 1992 --
  1997, Jul. 2015.

\bibitem{Xinlei}
X.~Yi, ``Resource-constrained multi-agent control systems: Dynamic
  event-triggering, input saturation, and connectivity preservation,''
  Licentiate Thesis, Royal Institute of Technology, Sweden, 2017.

\end{thebibliography}
\bibliographystyle{IEEEtran}

\end{document}